\documentclass[11pt,a4paper]{article}

\usepackage[ngerman,english]{babel}

\usepackage{hyperref}
\hypersetup{colorlinks,citecolor=blue,filecolor=blue,linkcolor=blue,urlcolor=blue}
\usepackage{latexsym,amsmath,amstext,amsthm,amscd,amsopn,amssymb,verbatim,amsrefs}
\usepackage{graphicx}
\usepackage{geometry}
\usepackage[dvips]{color}
\usepackage{epsfig}
\usepackage[all]{xy}
\usepackage{amscd}
\usepackage{amssymb}
\usepackage{latexsym}

\DeclareMathOperator{\id}{id}

\newtheorem{satz}{Satz}[subsection]

\newtheorem*{theorem*}{Theorem}

\newtheorem{proposition}[satz]{Proposition}

\newtheorem*{definition*}{Definition}
\renewcommand{\abstract}{\textbf{Abstract}}

\date{}

\begin{document}
\title{\vspace{-0.5cm}\textbf{\textit{On the $\mathfrak{grt}$ hexagon symmetry}}}\vspace{-0.8cm}
\author{\textit{Johannes L\"offler}}\vspace{-0.7cm}
\maketitle
\begin{abstract} In this paper we show that it is possible to project onto the solutions of the $\mathfrak{grt}$ hexagon equation. We also consider in some sense generalized hexagon equations and other symmetry equations for multiple argument maps between groups or torsors as source and show that for these equations we can construct at least some canonical solutions by symmetrization procedures. With help of this solutions we finally introduce a bunch of differentials associated to the considered symmetries.
\end{abstract}
\vspace{-0.2cm}\subsubsection*{\textit{Introduction}}
The Grothendieck-Teichm\"uller group was originally introduced by Drinfeld in his seminal article \cite{Dri}, his studies of deformations of quasi-triangular quasi-Hopf quantized universal enveloping algebras, as the set of degenerate associators. Associators also appear in knot theory and there are deep connections to number theory. 
Let $\mathfrak{lie}_2=\mathbb{K}\ll\hspace{-0.1cm}{x,y}\hspace{-0.1cm}\gg$ denote the completed free lie algebra in two generators. The Drinfeld-Kohno Lie algebra $\mathfrak{t}_n$ for $n\geq2$ is generated by symbols $t_{ij}=t_{ji}$, $1\leq{i,j}\leq{n}$, $i\neq{j}$ with the relations
$[t_{ij},t_{kl}]=0\;\text{if}\;\vert\{i,j,k,l\}\vert=4$ and 
$[t_{ij},t_{ik}+t_{kj}]=0\;\text{if}\;\vert\{i,j,k\}\vert=3$. The Grothendieck-Teichm\"uller Lie algebra $\mathfrak{grt}$ is defined as the vector space of solutions $\phi(x,y)\in\mathfrak{lie}_2$ of the equations
$\phi(x,y)=-\phi(y,x)$ (skew-symmetry)
\begin{equation}\label{Hexa}
\phi(x,y)+\phi(y,-x-y)+\phi(-x-y,x)=0
\end{equation}
\vspace{-0.7cm}\begin{equation}\label{Penta}
\phi(t_{12},t_{23}+t_{24})+\phi(t_{13}+t_{23},t_{34})\hspace{-0.1cm}=\hspace{-0.1cm}\phi(t_{23},t_{34})+\phi(t_{12}+t_{13},t_{24}+t_{34})+\phi(t_{12},t_{23})
\end{equation}
The equation \ref{Hexa} is called the hexagon and \ref{Penta} the pentagon. As the name suggests there is a Lie algebra structure on solutions of the previous equations called the Ihara bracket \cite{Iha}: For $f\in\mathfrak{lie}_2$ we define a derivation $D_f$ of $\mathfrak{lie}_2$ by setting
$D_f(x)=0$,
$D_f(y)=[y,f]$
and the Ihara bracket $\{\cdot,\cdot\}$ is defined by
$\{f,g\}:=[f,g]+D_f(g)-D_g(f)$.

Kontsevich's formality theorem \cite{K} yields the most general approach to deformation quantization. Tamarkin gave an independent proof of the formality theorem for $\mathbb{R}^d$ with operadic methods \cite{TaF}, the globalisation of this result has been established by Halbout \cite{Ha}. Drinfeld associators are essential in Tamarkin's proof of formality of little discs, a step in his operadic formality proof. Willwacher proved that the Grothendieck-Teichm\"uller Lie algebra $\mathfrak{grt}$ is isomorphic to the zeroth cohomology of the Kontsevich graph complex, considered as a Lie algebra \cite{TRT}. A result of Dolgushev states that the action of the Grothendieck-Teichm\"uller group $\mathrm{GRT}$ on the set of homotopy classes of stable formality quasi-isomorphisms is transitive and faithful \cite{DoSt}. Willwacher \cite{WiGRT} upgraded this result to the more general Shoikhet-Tsygan formality maps of chains, see \cite{ShTsy}.


{\textbf{Acknowledgements}:} I am deeply grateful to H. Furusho for explanations and useful comments that motivated some generalizations. I like to thank R. Friedrich, R.-M. Kaufmann, C. Mautner, S. Merkulov and T. Willwacher for discussions, especially D. Radchenko for pointing out the reference \cite{Khr} and G. Schaumann for his comment that an observation does not restrict to the abelian case. Last but not least I thank the MPIM \textsc{Max-Planck-Institute For Mathematics} in Bonn for hospitality and financial support. johannes@mpim-bonn.mpg.de
\section{\textit{Projective solution of the $\mathfrak{grt}$ hexagon equation}}

An important result of Furusho \cite{Fu} is that the solutions of the pentagon equation without quadratic term satisfy skew-symmetry and also the hexagon, see also \cite{TRT} and \cite{DBN} for alternative proofs of this highly non-trivial statement. Hence skew-symmetry and hexagon are not really necessary to define $\mathfrak{grt}$, but nevertheless it is clear that we can project onto the solutions of the first condition by skew-symmetrization and we claim that it is also easy to project onto the solutions of the hexagon, in representation theory this result can also be seen as a special case of a procedure known as Young projectors:
\begin{proposition}\label{Hexagon} The operator $\mathcal{H}:\mathfrak{lie}_2\rightarrow\mathfrak{lie}_2$ defined by the formula
$$(\mathcal{H}\varphi)(x,y):=[2\varphi(x,y)-\varphi(y,-x-y)-\varphi(-x-y,x)]/3$$
projects onto solutions of the hexagon equation.
\end{proposition}
\begin{proof} We construct the solutions as a linear combination of certain pre-compositions: First we define an operator $\alpha:\mathfrak{lie}_2\rightarrow\mathfrak{lie}_2$ by the formula
$(\alpha\varphi)(x,y):=\varphi(y,-x-y)+\varphi(-x-y,x)$.
With the substitution $\alpha$ the hexagon reads
$\phi+\alpha\phi=0$.

For the square $\alpha^2$ we calculate 
$$(\alpha^2\varphi)(x,y)=(\alpha\varphi)(y,-x-y)+(\alpha\varphi)(-x-y,x)$$
$$=\varphi(-x-y,x)+\varphi(x,y)+\varphi(x,y)+\varphi(y,-x-y)$$
hence $\alpha^2=2\id+\alpha$.
Now consider for $\lambda,\beta\in\mathbb{C}$ the operators $\id+\lambda\alpha$ and $\id+\beta\alpha$: It is not difficult to verify
$\bigr(\id+{\beta}\alpha\bigr)\circ\bigr(\id+\lambda\alpha\bigr)=(1+2\lambda\beta)\id+(\lambda+\beta+\lambda\beta)\alpha$.

The previous calculation shows that for all $\lambda\neq1,-1/2$ the operators $\id+\lambda\alpha$ are invertible and also that the Ansatz
$\varphi_\lambda=\varphi+\lambda\alpha\varphi$
is good to produce solutions of $\phi+\alpha\phi=0$, for instance we compute
$$\varphi_\lambda+\alpha\varphi_\lambda=\varphi+\lambda\alpha\varphi+\alpha\varphi+\lambda\alpha^2\varphi=(1+2\lambda)\left(\alpha\varphi+\varphi\right)$$
Hence for every $\varphi$ we can define a solution of $\phi+\alpha\phi=0$ by $\phi_\varphi:=\varphi_{-1/2}=\varphi-\alpha\varphi/2$. If $\varphi$ solves $\varphi+\alpha\varphi=0$ we have
$\varphi_{-1/2}=3\varphi/2$
and finally verified that the normalized operator
$\mathcal{H}:=\left(2\id-\alpha\right)/3:\mathfrak{lie}_2\rightarrow\mathfrak{lie}_2$
projects onto the solutions of $\phi+\alpha\phi=0$.\end{proof}
As mentioned in the previous proof
$\phi(x,y)+\lambda\phi(-x-y,x)+\lambda\phi(y,-x-y)=\forall{x,y}$ admits no non-trivial solutions if $\lambda\neq1,-1/2$ and $\mathcal{A}:\mathfrak{lie}_2\rightarrow\mathfrak{lie}_2$ defined by
$(\mathcal{A}\phi)(x,y):=[\phi(x,y)+\phi(y,-x-y)+\phi(-x-y,x)]/3$
projects onto the solutions of the anti-hexagon equation
$\phi(x,y)-\phi(y,-x-y)/2-\phi(-x-y,x)/2=0\;\forall{x,y}$.

The operator $\mathcal{H}$ is compatible with skew-symmetrization and notice that Drinfeld in his original article \cite{Dri} describes $\mathfrak{grt}$ with the additional equation
\begin{equation}\label{DAE}
[y,\phi(x,y)]+[z,\phi(x,z)]=0\end{equation}
if again the mass shell constraint ${x+y+z=0}$ is satisfied. Drinfeld in \cite{Dri} shows that skew-symmetry combined with hexagon and pentagon imply \ref{DAE} and Ihara shows in \cite{Iha} that skew-symmetry and \ref{DAE} imply the hexagon, hence the projection onto the hexagon solutions and \ref{DAE} or the more involved pentagon equation are compatible.

For abelian groups we have a square root of the map $f(x,y)=(y,-x-y)$, namely the map $(x,y)\rightarrow(x+y,-x)$ squares to $f$. 

The proof of \ref{Hexagon} in fact only uses group structures to produce hexagon solutions and some parts of \ref{Hexagon} can be deduced from the following proposition:

\begin{proposition}\label{BH}
Let $X$ be a set equipped with a map $f:X\rightarrow{X}$ that satisfies the composition relation $f^3=\id$, in other words $f$ represents the cyclic group $\mathbb{Z}_3$. For any map $\varphi:X\rightarrow{G_t}$, where $(G_t,\circ_t,\mathrm{e}_t)$ is a group, the operation
$$\phi_\varphi:=\left(\varphi\circ{f}\right)^{-1}\circ_t\varphi\circ_t\left(\varphi\circ{f^2}\right)^{-1}\circ_t\varphi:{X}\rightarrow{G_t}$$
is a solution of the hexagon equation
$$(\phi\circ{f})\circ_t\phi\circ_t(\phi\circ{f^2})=\mathrm{e}_t$$
\end{proposition} 
\begin{proof} The proof is just a direct check by calculation, we thank G. Schaumann for pointing out that we do not have restrict to an abelian source group in a draft version.\end{proof}
 As pointed out by H. Furusho the solutions in \ref{BH} in some sense look quite similar to a formula in \cite{AT2} where solutions of the Kashiwara-Vergne conjecture are linked to solutions of a pentagon like equation, in this construction of pentagon like solutions another essential ingredient is the associativity of the Baker-Campbell-Hausdorff series.


Notice that we do not claim in \ref{BH} that we get all solutions of the non-abelian hexagon between arbitrary groups, we simplify a bit: For example for every $\varphi:G\rightarrow{G}$ the map $\phi_\varphi(x):=\varphi(x)\circ\varphi^{-1}(x^{-1})$ satisfies the parity equation $\phi_\varphi(x)=\phi_\varphi^{-1}(x^{-1})$ but if $\varphi(x)=\varphi^{-1}(x^{-1})$ then $\phi_\varphi(x):=\varphi(x)\circ\varphi(x)$ and not every group admits square roots. For any binary operation $\varphi:G_s\times{G_s}\rightarrow{G_t}$, where $(G_s,\circ_s,1_s)$ and $(G_t,\circ_t,1_t)$ are groups, the operations $\sigma_\varphi,\tilde{\sigma}_\varphi:G_s\times{G_s}\rightarrow{G_t}$ defined by
$\sigma_\varphi(x,y):=\varphi(x,y)\circ_t\varphi(y,x)^{-1}$
and $\tilde{\sigma}_\varphi(x,y):=\sigma_\varphi(x,y)^{-1}$ obviously solve the skew-symmetry equation
$\sigma(x,y)=\sigma(y,x)^{-1}$
but if we consider arbitrary non-abelian groups the ``symmetric" equation $\sigma(x,y)=\sigma(y,x)$ does not admit an obvious symmetrization procedure. From this point of view the hexagon is in some sense quite natural.

\subsection{Some representations of $\mathbb{Z}_5$ and $\mathbb{Z}_{n+1}$}

\begin{proposition}\label{GH}
Let ${n}\in\mathbb{N}^+$, $(G_s,+,0)$ and $(G_t,+,0)$ be abelian groups and consider an operation
$\varphi:G_s^{\times{n}}\rightarrow{G_t}$ with the property that $\varphi$ is totally symmetric or skew-symmetric, {\em i.e.}
$\varphi(x_{\sigma(1)},\cdots,x_{\sigma(n)})=(\pm)^{M(\sigma)}\varphi(x_1,\cdots,{x}_n)$ respectively, where $M(\sigma)$ is the number of transpositions. The maps
$\phi_\varphi,\Phi_\varphi:G_s^{\times{n}}\rightarrow{G_t}$
defined by
\vspace{-0.3cm}\begin{align*}
&\phi_\varphi(x_1,\cdots,{x}_n):=n\varphi(x_1,\cdots,{x}_n)\mp\sum_{i=1}^{n}\varphi\Biggr(x_1,\cdots{x}_{i-1},-\sum_{j=1}^{n}x_j,x_{i+1},\cdots,{x}_n\Biggr)
\\&\Phi_\varphi(x_1,\cdots,{x}_n):=\varphi(x_1,\cdots,{x}_n)\pm\sum_{i=1}^{n}\varphi\Biggr(x_1,\cdots{x}_{i-1},-\sum_{j=1}^{n}x_j,x_{i+1},\cdots,{x}_n\Biggr)
\end{align*}
are respectively solutions of the hexagon or anti-hexagon equation
\begin{align*}
&\phi(x_1,\cdots,{x}_n)\pm\sum_{i=1}^{n}\phi\Biggr(x_1,\cdots{x}_{i-1},-\sum_{j=1}^{n}x_j,x_{i+1},\cdots,{x}_n\Biggr)=0\quad\forall{x_1,\cdots,x_n}\in{G}_s\\&n\Phi(x_1,\cdots,{x}_n)\mp\sum_{i=1}^{n}\Phi\Biggr(x_1,\cdots{x}_{i-1},-\sum_{j=1}^{n}x_j,x_{i+1},\cdots,{x}_n\Biggr)=0\quad\forall{x_1,\cdots,x_n}\in{G}_s
\end{align*}
\end{proposition}
\begin{proof} The proof is just a direct check where we only transpose in some terms once. If $G$ is any non-commutative group the upgrade of the map implicit present in \ref{GH} to a representation of a $n+1$ cycle by a map $G^{\times{n}}:\rightarrow{G}^{\times{n}}$ is the following:
$$(x_1,\cdots,x_n)\rightarrow{P}(x_1,\cdots,x_n):=({x_2}^{-1}\circ\cdots\circ{x_n}^{-1}\circ{x_1}^{-1},x_3,\cdots,x_n,x_1)$$
always satisfies $P^{n+1}=\id$ and hence represents $\mathbb{Z}_{n+1}$, explicit the other iterations are determined by ${P}^2(x_1,\cdots,x_n)=(x_{2},x_{4},x_{5},\cdots,x_n,x_1,{x_{}}^{-1}\circ\cdots\circ{x_n}^{-1}\circ{x_1}^{-1})$, 
${P}^l(x_1,\cdots,x_n)=(x_{l},x_{l+2},x_{l+3},\cdots,x_n,x_1,{x_{}}^{-1}\circ\cdots\circ{x_n}^{-1}\circ{x_1}^{-1},x_2,x_3,\cdots,x_{l-1})$ if $3\leq{l}\leq{n-1}$ and ${P}^n(x_1,\cdots,x_n)=(x_{n},{x_{}}^{-1}\circ\cdots\circ{x_n}^{-1}\circ{x_1}^{-1},x_{2},x_{3},\cdots,x_{n-1})$.

If we suppose $G_t$ is abelian and $0=a_0+a_1+a_2+a_3+\cdots+a_n$ the map $\phi_\varphi:G_s^{\times{n}}\rightarrow{G_t}$ defined by 
$$\phi_\varphi=a_0\varphi+a_1\varphi\circ{P}+a_2\varphi\circ{P}^2+a_3\varphi\circ{P}^3+\cdots+a_n\varphi\circ{P}^n$$
for any $\varphi:G_s^{\times{n}}\rightarrow{G_t}$ is a solution of
$0=\phi+\phi\circ{P}+\phi\circ{P}^2+\phi\circ{P}^3+\cdots+\phi\circ{P}^n$.\end{proof}
For $n=1$ this are just the two parity conditions and for $n=2$ the two hexagons in the abelian case. The parity equation $\rho(x^{-1})=\rho(x)^{-1}$ has also another natural obvious generalization to $n$-ary operations in the non-abelian case: Let ${n}\in\mathbb{N}^+$ and $M$ be a subset of $\{1,\cdots,n\}$ and for $1\leq{i}\leq{n}$ the functions $f^M_i$ operate on a group by
$f^M_i(x)=\begin{cases}x^{-1}\quad\text{if}\;i\in{M}\\x\quad\quad\text{else}\end{cases}$.
For example for any
$\varphi:G_s^{\times{n}}\rightarrow{G_t}$ the operations
$\rho^M_\varphi,\tilde{\rho}^M_\varphi:G_s^{\times{n}}\rightarrow{G_t}$
defined by
$\rho^{M}_\varphi(x_1,\cdots,x_n):=\varphi\left(x_1,\cdots,x_n)\circ_t\varphi(f^M_i(x_1),\cdots,f^M_i(x_n)\right)$
and $\tilde{\rho}^M_\varphi(x_1,\cdots,x_n):=\rho_\varphi(x_1,\cdots,x_n)^{-1}$
satisfy
$\rho(x_1,\cdots,x_n)=\rho\left(f^M_i(x_1),\cdots,f^M_i(x_n)\right)$
but again the ``symmetric" equation $\rho(x_1,\cdots,x_n)=\rho\left(f^M_i(x_1),\cdots,f^M_i(x_n)\right)$ does not always admit any non-constant canonical solution procedure if the set $M$ is non-empty.

It is speculative but the aforementioned pentagon equation \ref{Penta} could be related to $\mathbb{Z}_5$ or a semigroup with $5$ elements. If $n+1$ is a prime number there is no representation of $\mathbb{Z}_{n+1}$ as automorphisms $G^{\times{m}}\rightarrow{G}^{\times{m}}$ for any $m<n$ as a consequence of \cite{Khr}, but we are not aware of an analogous result for semigroups or torsors. However for a field $\mathbb{K}$ it is more or less false that there are no five cycles $\mathbb{K}^{\times{m}}\rightarrow\mathbb{K}^{\times{m}}$ if $m<4$:

Consider the map
$f:\mathbb{K}^{\times2}\setminus\mathcal{I}\rightarrow\mathbb{K}^{\times2}\setminus\mathcal{I}$ defined by $f(x,y):=\left(y,\frac{1-x}{1-xy}\right)$ where $\mathbb{K}$ is a field and $\mathcal{I}=\{(0,x),(x,0),(1,x),(x,1),(x,1/x)\}\subset\mathbb{K}^{\times2}$,
this definition is motivated by the five term-relation of the Bloch-Wigner function. We have the iterations $f^2(x,y)=\left(\frac{1-x}{1-xy},1-xy\right)$, $f^3(x,y)=\left(1-xy,\frac{1-y}{1-xy}\right)$, $f^4(x,y)=\left(\frac{1-y}{1-xy},x\right)$ and $f^5=\id$.

If $D:\mathbb{K}\rightarrow\mathbb{K}$ satisfies $D(x)+D(y)+D\left(\frac{1-x}{1-xy}\right)+D\left(\frac{1-y}{1-xy}\right)+D(1-xy)=0$ then we also have for $D_\pm:\mathbb{K}^{\times2}\rightarrow\mathbb{K}$ defined by $D_\pm(x,y)=D(x)\pm{D}(y)$ the vanishing $\sum_{i=1}^5{D}_\pm\circ{f}^i=0$ and $\varphi\rightarrow(4\varphi-\sum_{i=1}^4\varphi\circ{f}^i)/5$ projects onto solutions of this symmetry.

\section{\textit{Square zero maps associated to Hexagon symmetries}}
There is a bunch of cohomologies associated to the hexagon symmetries, {\em i.e.} maps $f,\tilde{f}:X\rightarrow{X}$ with the composition relations $f^2=\tilde{f}$,$\tilde{f}^2=f$ and $f\circ\tilde{f}=\id=\tilde{f}\circ{f}$, we will give some examples we could figure out.
The following proposition considers $X=G^{\times2}$ and provides us with a map
$\partial:\bigr\{[\cdot,\cdot]:G^{\times2}\rightarrow{G}\;\text{bilinear}\bigr\}\stackrel{\text{linear}}{\rightarrow}\Bigr\{{\text{linear square zero maps}:\{\text{maps}:G^{\times{n}}\rightarrow{G}\}}\hspace{-0.15cm}\xymatrix
{\tiny\ar@(dr,ur)}\quad\;\;\Bigr\}$.
\begin{proposition}\label{1D}
Let $(G,+)$ be an abelian group and
$[\cdot,\cdot]:G^{\times2}\rightarrow{G}$
be a group homomorphism in both arguments, {\em i.e.} $[a+b,c]=[a,c]+[b,c]$ and $[a,b+c]=[a,b]+[a,c]$ holds. If $n\geq2$ we have a linear square zero map
$\partial_{[\cdot,\cdot]}:\{\text{symmetric maps}:G^{\times{n}}\rightarrow{G}\}\hspace{-0.15cm}\xymatrix
{\tiny\ar@(dr,ur)}\quad\;\;$
\begin{align*}
\left(\partial_{[\cdot,\cdot]}\psi\right)(x_1,\cdots,{x}_n)=\Biggr[\sum_{i=1}^{n}x_i,&\psi(x_1,\cdots,x_n)+\sum_{i=1}^{n}\psi\circ{f}^i\left(x_1,\cdots,x_n\right)\Biggr]
\end{align*}
where $f\left(x_1,\cdots,x_n\right)=\left(-\sum_{l=1}^{n}x_l,x_3,x_4,\cdots,\cdots,x_n,x_1\right)$. $\partial_{[\cdot,\cdot]}$ restricts to symmetric maps, here we have explicit the formula $\left(\partial_{[\cdot,\cdot]}\psi\right)(x_1,\cdots,{x}_n)=\big[\sum_{i=1}^{n}x_i,\psi(x_1,\cdots,x_n)+\sum_{i=1}^{n}\psi\left(x_1,\cdots,x_{i-1},-\sum_{l=1}^{n}x_l,x_{i+1}\cdots,x_n\right)\big]$.

If we consider the two maps $f(x,y)=\left(-x-y,x\right)$ and $\tilde{f}(x,y)=\left(y,-x-y\right)$ parity is conserved, {\em i.e.}
$\psi(x,y)=\pm\psi(y,x)\Rightarrow(\partial_{[\cdot,\cdot]}\psi)(x,y)=\pm(\partial_{[\cdot,\cdot]}\psi)(y,x)$.
\end{proposition}
\begin{proof} The proof that $\partial_{[\cdot,\cdot]}$ squares to zero is again just a straight forward calculation where for the $n$-ary maps we only transpose in some terms appearing in $\partial^2_{[\cdot,\cdot]}$ once and use the assumption that $[\cdot,\cdot]:G^{\times{2}}\rightarrow{G}$ is a group homomorphism in its two arguments. \end{proof}

There are also some non-linear square zero maps that use the hexagon symmetries.
\begin{proposition}\label{2D} Let $(G,+)$ be an abelian group and
$[\cdot,\cdot]:G^{\times2}\rightarrow{G}$
a group homomorphism in both arguments. The map
$\partial_{[\cdot,\cdot]}:\{maps:{G^{\times2}}\rightarrow{G}\}\hspace{-0.15cm}\xymatrix
{\tiny\ar@(dr,ur)}\quad\;\;$
defined by
$$\partial_{[\cdot,\cdot]}\psi=[\psi,\psi\circ{f}]+[\psi\circ{f},\psi\circ\tilde{f}]+[\psi\circ\tilde{f},\psi]$$
for $\psi:{X}\rightarrow{G}$ squares to zero
$\partial_{[\cdot,\cdot]}^2=0$.\end{proposition}
\begin{proposition}\label{3D}
Let $(G,\circ,\mathrm{e})$ be a group and suppose the binary operation $[\cdot,\cdot]:G^{\times{2}}\rightarrow{G}$ is skew-symmetric, {\em i.e.} $[b,a]=[a,b]^{-1}$ and a group homomorphism in each argument, {\em i.e.} we assume $[a\circ{b},c]=[a,c]\circ[b,c]$ and $[a,{b}\circ{c}]=[a,b]\circ[a,c]$. We have a square $\mathrm{e}$ map $\partial_{[\cdot,\cdot]}:\{maps:G^{\times2}\rightarrow{G}\}\hspace{-0.15cm}\xymatrix
{\tiny\ar@(dr,ur)}\quad\;\;$ defined by
$$\left(\partial_{[\cdot,\cdot]}\psi\right):=\left[\psi,(\psi\circ\tilde{f})\circ(\psi\circ{f})\right]$$
\end{proposition}
\begin{proof} Notice that the assumption that $[\cdot,\cdot]$ is a group homomorphism in both arguments and skew-symmetric implies $[a^{-1},b]=[a,b^{-1}]=[a,b]^{-1}=[b,a]$ and $[a,a]=[a,\mathrm{e}]=[\mathrm{e},a]=[a,a^{-1}]=\mathrm{e}$. The proof that $(\partial^2_{[\cdot,\cdot]}\psi)=\mathrm{e}$ is a straight forward calculation, let us write $A=\psi$, $B=\psi\circ\tilde{f}$ and $C=\psi\circ{f}$.
In fact we have under the assumptions the formula
$[a,c]\circ[b,c]=[c,a\circ{b}]^{-1}=[b,c]\circ[a,c]$
hence $[\cdot,\cdot]$ maps to an abelian sub-group.\end{proof}

\section{\textit{Some canonical solvable torsor symmetries}}
Consider a torsor $(X,\tau)$, {\em i.e.} a non-empty set $X$ endowed with a ternary operation
$\tau:X^{\times3}\rightarrow{X}$
that enjoys the following two reflection properties
\begin{equation}\label{Tor1}
\tau(x,y,y)=x=\tau(y,y,x)\quad\forall{x,y,z}\in{X}
\end{equation}
\vspace{-0.4cm}\begin{equation}\label{Tor2}
\tau(\tau(x,y,z),v,w)=\tau(x,y,\tau(z,v,w))\quad\forall{x,y,z,v,w}\in{X}
\end{equation}
Informally a torsor behaves like a group but the notion is more intrinsic, we do not have a preferred identity element in $X$. The standard example of a torsor is a group $G$ as set $X$ and $\tau$ is defined by
$\tau(x,y,z):=x\circ{y}^{-1}\circ{z}$.
\begin{proposition}
For any
$\varphi:X^{\times3}\rightarrow{G}$
where $(X,\tau)$ is a torsor and $(G,\circ,\mathrm{e})$ is a group, the four ternary operations
$\gamma^{\pm}_\varphi,\tilde{\gamma}^{\pm}_\varphi:X^{\times3}\rightarrow{G}$
defined by
$$\gamma^{-}_\varphi(x,y,z):=\varphi\Bigr(\tau(x,y,z),z,y\Bigr)^{-1}\circ\varphi\Bigr(y,x,\tau(x,y,z)\Bigr)$$
$$\gamma^{+}_\varphi(x,y,z):=\varphi\Bigr(\tau(x,y,z),z,y\Bigr)\circ\varphi\Bigr(y,x,\tau(x,y,z)\Bigr)^{-1}$$
and $\tilde{\gamma}^{\pm}_\varphi(x,y,z):=\gamma^{\pm}_\varphi(x,y,z)^{-1}$
are solutions of the equation
\begin{equation}\label{TS}
\gamma\Bigr(\tau(x,y,z),z,y\Bigr)\circ\gamma\Bigr(y,x,\tau(x,y,z)\Bigr)=\mathrm{e}\quad\forall{x,y,z}\in{X}
\end{equation}
\end{proposition}
\begin{proof}
The proof is straight forward with help of \ref{Tor1} and \ref{Tor2}, we will just prove it for $\gamma^{-}_\varphi$:
$$\gamma^{-}_\varphi(\tau(x,y,z),z,y\Bigr)=\varphi\Bigr(\tau(\tau(x,y,z),z,y),y,z\Bigr)^{-1}\circ\varphi\Bigr(z,\tau(x,y,z),\tau(\tau(x,y,z),z,y\Bigr)$$
$$=\varphi\Bigr(\tau(x,y,\tau(z,z,y)),y,z\Bigr)^{-1}\circ\varphi\Bigr(z,\tau(x,y,z),\tau(x,y,\tau(z,z,y)\Bigr)$$
$$=\varphi\Bigr(\tau(x,y,y),y,z\Bigr)^{-1}\circ\varphi\Bigr(z,\tau(x,y,z),\tau(x,y,y)\Bigr)$$
$$=\varphi\Bigr(x,y,z\Bigr)^{-1}\circ\varphi\Bigr(z,\tau(x,y,z),x\Bigr)$$
$$\gamma^{-}_\varphi(y,x,\tau(x,y,z)\Bigr)=\varphi\Bigr(z,\tau(x,y,z),x\Bigr)^{-1}\circ\varphi\Bigr(x,y,z\Bigr)$$
The calculation is a bit more transparent if we define $f_{1,2,3}:X^{\times3}\rightarrow{X}^{\times3}$ by $f_1(x,y,z)=(\tau(x,y,z),z,y)$, ${f}_2(x,y,z)=(y,x,\tau(x,y,z))$ and ${f}_3(x,y,z)=(z,\tau(x,y,z),x)$.
The previous computations showed $f_1^2=\id={f}_2^2$ and
$(f_1\circ{f}_2)(x,y,z)=(z,\tau(x,y,z),x)=f_3(x,y,z)=({f}_2\circ{f}_1)(x,y,z)$, hence $(\gamma^{\pm}_\varphi\circ{f}_1)\circ(\gamma^{\pm}_\varphi\circ{f}_2)=\mathrm{e}$. Notice if our torsor comes from an abelian group, {\em i.e.} $\tau(x,y,z)=\tau(z,y,x)$ then $f_3^2=\id$ and $f_2\circ{f}_3=f_1$, $f_1\circ{f}_3=f_2$ and hence $\{\id,f_1,f_2,f_3\}$ have the symmetries of the \textbf{Klein four-group}.
\end{proof}
Let $(G,+)$ be an abelian group. The $n$-ary operations $G^{\times{n}}\rightarrow{G}$ have naturally the structure of an abelian group by point wise addition. If $G$ is endowed with a Lie bracket $[\cdot,\cdot]$ then we can define an induced bracket on $n$-ary maps by setting $[\psi_1,\psi_2](x_1,\cdots,x_n)=[\psi_1(x_1,\cdots,x_n),\psi_2(x_1,\cdots,x_n)]$. Notice that although if $[\cdot,\cdot]$ is a Lie bracket, {\em i.e.} skew-symmetric and satisfies the Jacobi identity $[x,[y,z]]+[y,[z,x]]+[z,[x,y]]=0$ the map $\partial{[\cdot,\cdot]}$ is unfortunately not a derivation of the Lie bracket induced on $n$-ary maps.
\begin{proposition}
Let $(X,\tau)$ be a torsor and $(G,\circ,\mathrm{e})$ a group equipped with a skew-symmetric binary map $[\cdot,\cdot]:G^{\times2}\rightarrow{G}$ with $[b,a]=[a,b]^{-1}$. We have a square constant $\mathrm{e}$ map $\partial_{[\cdot,\cdot]}:\{maps:X^{\times3}\rightarrow{G}\}\hspace{-0.15cm}\xymatrix
{\tiny\ar@(dr,ur)}\quad\;\;$ defined by
$$(\partial_{[\cdot,\cdot]}\varphi)(x,y,z)=\left[\varphi\left(\tau(x,y,z),z,y),\varphi(y,x,\tau(x,y,z)\right)\right]$$
\end{proposition}
\begin{proposition}
Let $(X,\tau)$ be a torsor and $(G,+,0)$ be an abelian group equipped with a bilinear binary map $[\cdot,\cdot]:G^{\times2}\rightarrow{G}$. For every $\gamma:X^{\times3}\rightarrow{G}$ that solves $\gamma{f}_1+\gamma{f}_2=0$ we have two square zero maps $\partial^\gamma_{[\cdot,\cdot]}:\{maps:X^{\times3}\rightarrow{G}\}\hspace{-0.15cm}\xymatrix
{\tiny\ar@(dr,ur)}\quad\;\;$ by
$$\partial^\gamma_{[\cdot,\cdot]}\varphi=\left[\gamma,\varphi\circ{f}_1\pm\varphi\circ{f}_2\right]$$
and we have also $\partial^\gamma_{[\cdot,\cdot]}\varphi{f}_1\pm\partial^\gamma_{[\cdot,\cdot]}\varphi{f}_2=0$. In the $+$-case: If $[\cdot,\cdot]$ is a Lie bracket the induced brackets on ternary operations satisfies the modified Leibniz rule
\begin{align}\label{Strange}
\partial^\gamma_{[\cdot,\cdot]}[\varphi,\phi+\phi\circ{f_3}]=\partial^\gamma_{[\cdot,\cdot]}[\varphi+\varphi\circ{f_3},\phi]=\partial^{\partial^\gamma_{[\cdot,\cdot]}\varphi}_{[\cdot,\cdot]}\phi-\partial^{\partial^\gamma_{[\cdot,\cdot]}\phi}_{[\cdot,\cdot]}\varphi
\end{align}
\end{proposition}
\begin{proof}
We have $f_1^2=f_2^2=\id$, $f_1\circ{f}_2=f_2\circ{f}_1=f_3$ and $f_3\circ{f}_1=f_2$, $f_3\circ{f}_2=f_1$. The rest is just a small calculation with the Jacobi identity and \ref{Strange} is equivalent to
\begin{align*}
\partial^\gamma_{[\cdot,\cdot]}[\varphi,\phi]=&\bigr[\partial^\gamma_{[\cdot,\cdot]}\varphi,\phi\circ{f}_1+\phi\circ{f}_2\bigr]+\bigr[\varphi\circ{f}_1+\varphi\circ{f}_2,\partial^\gamma_{[\cdot,\cdot]}\phi\bigr]\\&-\bigr[\gamma,[\varphi\circ{f}_1,\phi\circ{f}_2]+[\varphi\circ{f}_2,\phi\circ{f}_1]\bigr]\vspace{-1.3cm}
\end{align*}\end{proof}
\begin{proposition}
Let $(X,\tau)$ be a torsor. The map $\iota:X^{\times3}\rightarrow{X}^{\times3}$ defined by the formula $\iota(x,y,z):=(\tau(y,x,z),x,\tau(\tau(y,x,z),x,y))$ is a $3$-cycle, for instance we have the identities $\iota^2(x,y,z):=(y,\tau(y,x,z),\tau(y,\tau(y,x,z),x))$ and $f^3=\id$.
\end{proposition}
\begin{proof} For most other symmetries we could figure out we need to add to \ref{Tor1}, \ref {Tor2} the heap condition
$\tau(\tau(x,y,z),v,w)=\tau(x,\tau(v,z,y),w)$
but here the proof works without it.
\end{proof}

\vfill{\begin{center}
\textcircled{c} Copyright by Johannes L\"offler, 2014, All Rights Reserved
\end{center}}
\pagebreak
\end{document}